\documentclass[draft]{amsart}
\usepackage{latexsym,bm}
\usepackage{amsthm}
\usepackage{amsfonts}
\usepackage{amsmath}
\usepackage[all]{xy}
\usepackage[final]{graphics}
\usepackage{fancyhdr}
\usepackage{verbatim}
\usepackage{geometry}
\usepackage{xcolor}
\usepackage{enumerate}
\usepackage{amssymb}

\newtheorem{thm}{Theorem}[section]
\newtheorem{cor}[thm]{Corollary}
\newtheorem{lem}[thm]{Lemma}
\newtheorem{prop}[thm]{Proposition}

\theoremstyle{definition}

\newtheorem{rem}[thm]{Remark}

\newtheorem{que}[thm]{Question}

\numberwithin{equation}{section}

\newcommand{\lemref}[1]{Lemma~\ref{#1}}
\newcommand{\propref}[1]{Proposition~\ref{#1}}

\newcommand{\remref}[1]{Remark~\ref{#1}}

\renewcommand{\O}{\mathcal{O}}
\renewcommand{\L}{\mathcal{L}}

\newcommand{\PP}{\mathbb{P}}
\renewcommand{\P}{\mathcal{P}}
\newcommand{\bpi}{\overline{\pi}}

\newcommand{\CC}{\mathbb{C}}

\newcommand{\B}{\overline{B}}

\renewcommand{\O}{\mathcal{O}}
\renewcommand{\L}{\mathcal{L}}

\renewcommand{\l}{\mathrm{l}}
\renewcommand{\c}{\mathrm{c}}

\title[A two-dimensional family of surfaces of general type with $p_g=0$ and $K^2=7$]{A two-dimensional family of surfaces of general type \\with $p_g=0$ and $K^2=7$}

\date{}
\author{Yifan Chen}
\author{YongJoo Shin}

\begin{document}
\maketitle

\begin{abstract}
We study the construction of complex minimal smooth surfaces $S$ of general type with $p_g(S)=0$ and $K_S^2=7$. Inoue constructed  the first examples of such surfaces, which can be described as Galois $\mathbb{Z}_2\times\mathbb{Z}_2$-covers over the four-nodal cubic surface. Later the first named author constructed more examples as Galois $\mathbb{Z}_2\times\mathbb{Z}_2$-covers over certain six-nodal del Pezzo surfaces of degree one.

In this paper we construct a two-dimensional family of minimal smooth surfaces of general type with $p_g=0$ and $K^2=7$, as Galois $\mathbb{Z}_2\times\mathbb{Z}_2$-covers of certain rational surfaces
with Picard number three, with eight nodes and with two elliptic fibrations.
This family is different from the previous ones.
\end{abstract}

\section{Introduction}
Classifications and constructions of algebraic surfaces of general type are still a crucial area among algebraic geometry. There are several approaches for the classifications and the constructions (see \cite{WBCPAV84, IBFCRP11}).  In this paper,  we focus on minimal surfaces $S$ of general type with $p_g(S)=0$ and $K_S^2=7$.

Inoue \cite{MI94} constructed first examples of such surfaces, which can be described as  Galois $\mathbb{Z}_2\times\mathbb{Z}_2$-covers over the four-nodal cubic surface \cite[Example~4.1]{MendesPardini01}. These examples are called Inoue surfaces with $K_S^2=7$ (and called Inoue surfaces in this paper). Mendes Lopes and Pardini \cite{MendesPardini01} studied the bicaonical map of a surface $S$ of general type with $p_g=0$ and $K_S^2=7$ and showed that the bicanonical map of an Inoue surface has degree two.
Lee and the second named author \cite{YLYS14} provided all possible fixed loci  of an involution $\sigma$ on $S$ (see also \cite{CR15}). Especially,  they showed that there only two cases for the divisorial fixed part $R_\sigma$ when the quotient of $S$ by $\sigma$ is birational to an Enriques surface:
    \begin{enumerate}[\upshape (i)]
        \item $R_\sigma$ is a disjiont union of two smooth irreducible curves of genus $3$ and $1$ with self intersection numbers $0$ and $-1$ respectively.
        \item $R_\sigma$ is a smooth irreducible curve of genus $3$ with self intersection number $-1$.
    \end{enumerate}
Inoue surfaces realized case (i).
  Inspired by the construction of Inoue surfaces and the work of \cite{YLYS14}, the first named author
 studied pairs $(S, G)$, where $S$ is as above, $G \le \mathrm{Aut}(S)$ and $G$ is isomorphic to the Klein group $\mathbb{Z}_2\times \mathbb{Z}_2$ (see \cite{YC15}) and constructed a new family of examples, which are  Galois $\mathbb{Z}_2\times\mathbb{Z}_2$-covers over certain six-nodal del Pezzo surfaces of degree one (see \cite{YC13}).
 These surfaces realized  case (ii) and their bicanonical maps have degree one.

Set $G=\{\mathrm{Id}_S, g_1, g_2, g_3\}$ and let $R_i$ be the divisorial fixed part of $g_i$ for $i=1, 2, 3$. Then the first named author showed that there are only three possibilities for $(K_SR_1, K_SR_2, K_SR_3)$ (assuming $K_SR_1 \ge K_SR_2 \ge K_SR_3$)
$$\mathrm{(a)}\ (7, 5, 5);\ \mathrm{(b)}\ (5, 3, 3);\ \mathrm{(c)}\  (5, 3, 1)$$
(see \cite[Theorem 1]{YC15}).
The pairs $(S, G)$ in the case (a) are exactly the Inoue surfaces and
those in  the case (b) are exactly the ones constructed in \cite{YC13}.
 On the other hand the existence of the examples in the case (c) remained as an open problem.
  Recently, Rito \cite{CR18} constructed one example by using a computer program.

In this paper, we construct a two-dimensional family of surfaces in the case (c).
\begin{thm}[{cf. Remark \ref{rem:twodim}, Proposition \ref{prop:intermediate} and Corollary \ref{cor:birational}}]\label{thm:diffpara}
Let $S$ be a surface constructed in Section \ref{construction}. Then $\mathrm{Aut}(S)$ contains a subgroup
$G=\{\mathrm{Id}_S, g_1, g_2, g_3\} \cong \mathbb{Z}_2\times \mathbb{Z}_2$ and $(K_SR_1,K_SR_2,K_SR_3)=(5,3,1)$. Moreover, $K_S$ is ample and the bicanonical morphism of $S$ is birational.

 The surfaces $S$ are parametrized by a two-dimensional irreducible variety.
\end{thm}

For the pairs $(S, G)$ in the cases (a) and (b), $S/G$ is a (singular) del Pezzo surface.
But this is no longer true for the case (c), which is the main obstacle to construct examples in the case (c). In this case, $S/G$ is a rational surface with nodes and $K_{S/G}^2=-1$ \cite[Theorem~1]{YC15}.
Denote by $B_i$ the image of $R_i$ under the quotient map $S\rightarrow S/G$.
Under certain birational map $S/G \dashrightarrow \PP^2$,
$B_1$ and $B_2$ correspond very singular curves of degree $6$ and degree $7$ respectively.
It is very hard to find such curves directly. We study two elliptic fibrations on $S/G$ in detail
and succeed to show that $B_1$ and $B_2$ correspond to irreducible smooth components in certain singular fibers of elliptic fibrations. See Remark~\ref{rem:notcubic}.

Let $\mathcal{M}_{1,7}^{\textrm{can}}$ be the Gieseker moduli space with $\chi(\O_S)=1$ and $K_S^2=7$ (see \cite{DG77}).
Inoue surfaces yield a $4$-dimensional irreducible connected component of $\mathcal{M}_{1,7}^{\textrm{can}}$ (see \cite{IBFC12}). And similarly the surfaces constructed by
 the first named author yield another irreducible connected component, which is of dimension $3$ (see \cite{YC15}).
 Thus it arises that
 \begin{que}
      Do the surfaces constructed in Section \ref{construction} yield an irreducible connected component of
                $\mathcal{M}_{1,7}^{\textrm{can}}$?
 \end{que}
We could not answer this question at the moment. On the one hand, we do not know whether any pair $(S, G)$ in the case (c) is in the $2$-dimensional family above. On the other hand, the study of deformations of the surfaces in the cases (a) and (b) depends on the fact that $S/G$ is a Del Pezzo surface (see \cite[Proof of Theorem~5.1 and Lemma 5.5]{IBFC12} and \cite[Proof of Proposition~5.1]{YC15}).

Bauer \cite{IB14} showed that any Inoue surface $S$ satisfies the Bloch conjecture by showing that
the Kodaira dimension $\kappa(S/g_i)<2$ \cite[Propositions 3.2 and 3.6]{IB14} and by applying the method of enough automorphisms in \cite{RB85,IHMM79}.
In the same way, the first named author proved that a surface in the case (b) also satisfies the Bloch conjecture (see \cite[Section 5]{YC13}).

 The Kodaira dimensions of the intermediate double covers $S/g_i$ of the surface $S$ constructed in Section \ref{construction} are considered.
\begin{prop}[{cf. Proposition \ref{prop:intermediate}}]
Let $S$ be a surface constructed in Section \ref{construction}. Then $S/g_1$ with $K_{S/g_1}^2=-2$ is birational to an Enriques surface, and $S/g_2$ is birational to a minimal properly elliptic surface and the minimal resolution of $S/g_3$ is a numerical Campedelli surface, a minimal surface of general type with $p_g=0$ and $K^2=2$.
\end{prop}

\begin{que}\label{Q2}
Does a surface $S$ constructed in  Section~4 satisfy the Bloch conjecture?
\end{que}
\noindent By \cite[Proposition 1.3]{IB14} and the proposition above, the question \ref{Q2} is equivalent to
\begin{que}
    For a surface $S$ construced in Section~4, does the minimal model of $S/g_3$ satisfy the Bloch conjecture?
 \end{que}

\medskip

\paragraph{Notation and conventions}
Through this article we work over the field of complex numbers.  Denote by $\equiv$ a linear equivalence between divisors. A $(-m)$-curve with a nonnegative integer $m$ on a smooth projective surface means that the curve is rational and of a self intersection number $-m$. In particular $(-2)$-curve is called a nodal curve.

\section{Configuration of points}\label{confpoints}
Let $\P$ be a subset of $\CC\PP^2$ which consists of six distinct points $p_0, p_1, \ldots, p_4, p$.
For $k=1, 2, 3, 4$, denote by $p_k'$ the infinitely near point over $p_k$ which corresponds to the line $\l_{p_0p_k}$. Here $\l_{q\bar{q}}$ stands for the line passing two points $q, \bar{q} \in \CC\PP^2$.
Assume that $\P$ satisfies the following conditions:
\begin{enumerate}
\item[\upshape (I)] Any three points in $\P$ are not contained in a line.
\item[\upshape (II)] There is a smooth conic $\gamma_1$ passing through the six points
                     $$p_1, p_2, p_3, p_3', p_4, p_4',$$
                    and there is a smooth conic $\gamma_2$ passing through the six points
                        $$p_1, p_1', p_2, p_2', p_3, p_4.$$

                    Here a curve passes through $p_k'$ means that this curve is tangent to
                     $\l_{p_0p_k}$ at $p_k$.
\item[\upshape (III)] The six points in $\P$ are not contained in a conic and the ten points
                      $$p_0, p_1, \ldots, p_4, p, p_1', \ldots, p_4'$$
                      are not contained in a cubic.
\item[\upshape (IV)] There is an irreducible cubic $\lambda_1$ passing through the seven points
                    $$p_0, p_1, p_2, p_2', p_3, p_4, p_4',$$
                    and having a singular point at $p$, and
                    there is an irreducible cubic $\lambda_2$ passing through the seven points
                    $$p_0, p_1, p_1', p_2, p_3, p_3', p_4,$$
                     and having a singular point at $p$.
\end{enumerate}

We aim to show that the existence of $\P$  and figure out the configuration of the six points.
We postpone the proofs of \lemref{lem:p0}, \lemref{lem:gamma0} and \propref{prop:p} to Appendix.
The upshot is to find the point $p$ and to verify the condition (IV) (see \remref{rem:upshot}).

Assume $p_1, \ldots, p_4$ satisfy the condition (I) then we may assume
\begin{align*}
p_1=(1:0:0),\ p_2=(0:1:0),\ p_3=(0:0:1),\ p_4=(1:1:1).
\end{align*}
Here we use $(x_1:x_2:x_3)$ as coordinates of $\CC\PP^2$.
Set
\begin{align*}
&o:=\l_{p_1p_2}\cap \l_{p_3p_4}=(1:1:0), &&\\
&r_0:=\l_{p_1p_4}\cap \l_{p_2p_3}=(0:1:1), && s_0:=\l_{p_1p_3}\cap \l_{p_2p_4}=(1:0:1), \\
&q_1:=\l_{p_1p_2}\cap \l_{r_0s_0}=(-1:1:0), && q_2:=\l_{p_3p_4}\cap \l_{r_0s_0}=(1:1:2).
\end{align*}
\begin{lem}\label{lem:p0}
The fives points $p_0, p_1, \ldots, p_4$ satisfy conditions {\rm (I)} and {\rm (II)} if and only if
 $p_0 \in \l_{r_0s_0}\setminus\{r_0, s_0, q_1, q_2\}$, i.e. $p_0=(t:1:1+t)$ and $t \not =0, \pm 1$.
\end{lem}

 \setlength{\unitlength}{1cm}
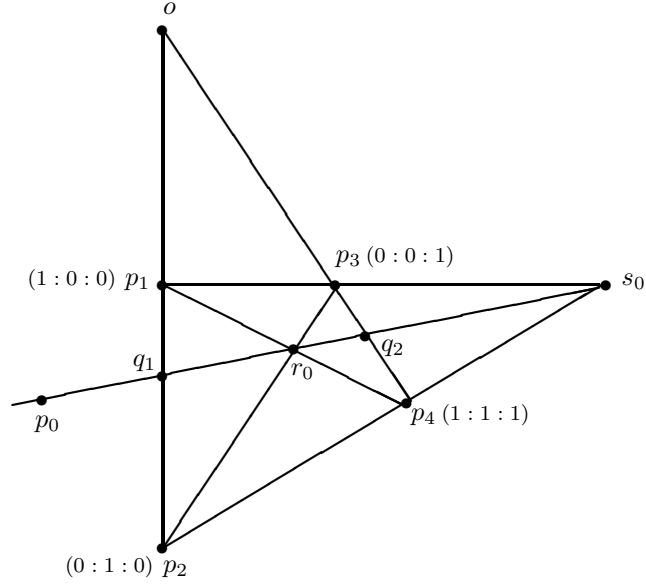
\begin{picture}(10,8)
  \thicklines
  \put(4,0.5){\line(5,3){5.9}}
  \put(4,4){\line(1,0){5.8}}
  \put(4,0.5){\line(0,1){6.9}}
  \put(4,4){\line(2,-1){3.2}}
  \put(2,2.4){\line(5,1){7.8}}
  \put(4,7.4){\line(2,-3){3.3}}
  \put(4,0.5){\line(2,3){2.3}}

   \put(4,7.6){$o$}
   \put(2.3,2.1){$p_0$}
   \put(3.5,4){$p_1$}
   \put(4,0.18){$p_2$}
   \put(6.3,4.3){$p_3$}
  \put(7.3,2.2){$p_4$}
    \put(3.6,2.92){$q_1$}
       \put(6.9,3.12){$q_2$}
    \put(5.7,2.8){$r_0$}
  \put(10.1,4){$s_0$}

  \put(2.2,4){{\footnotesize $(1:0:0)$}}
  \put(2.7,0.18){{\footnotesize $(0:1:0)$}}
  \put(6.7,4.3){{\footnotesize $(0:0:1)$}}
  \put(7.7,2.2){{\footnotesize $(1:1:1)$}}

  \put(3.9,7.3){$\bullet$}
   \put(2.3,2.37){$\bullet$}
  \put(3.9,3.9){$\bullet$}
  \put(3.9,0.4){$\bullet$}
   \put(6.2,3.9){$\bullet$}
   \put(7.15,2.33){$\bullet$}
   \put(3.9,2.7){$\bullet$}
   \put(6.6,3.22){$\bullet$}
   \put(5.65,3.05){$\bullet$}
   \put(9.8,3.9){$\bullet$}

\end{picture}

\begin{lem} \label{lem:gamma0} Assume the five points $p_0, p_1, p_2, p_3, p_4$ satisfy conditions {\rm (I)} and {\rm (II)}. Then they also satisfy the following property
\begin{itemize}
\item[(II')] There is a unique  cubic $\gamma_0$
passing through the nine points $p_0, p_1, p_1', \ldots, p_4, p_4'$.
\end{itemize}
\end{lem}

Let $p_1, \ldots, p_4$ as before. Observe that a smooth conic $\c_\alpha$ passing through these $4$ points has the defining equation
$$-(1+\alpha)x_1x_2+x_1x_3+\alpha x_2x_3=0, \alpha \not=0, -1.$$
For $\alpha=-t^2$, the conic $\c_{-t^2}$ also contains $p_0$. We denote it by $\c_{p_0p_1p_2p_3p_4}$.
\begin{prop}\label{prop:p}Let $p_1, \ldots, p_4$ be as before and fix $p_0=(t:1: 1+t)$ for $t \not=0, \pm 1$. Let $\c_{\alpha(t)}$ and $\c_{\beta(t)}$ be the smooth conics defined by
\begin{align*}
\c_{\alpha(t)}: -(1+\alpha(t))x_1x_2+x_1x_3+\alpha(t) x_2x_3=0,\\
\c_{\beta(t)}: -(1+\beta(t))x_1x_2+x_1x_3+\beta(t) x_2x_3=0,
\end{align*}
where $\alpha(t)+\beta(t)=-2t^2$, $\alpha(t)\beta(t)=t^2$, $t \not=0, \pm 1$.
Then the points $p_0, p_1, \ldots, p_4, p$ satisfy conditions {\rm (I)-(IV)} if and only if
$p \in (\c_{\alpha(t)}\cup \c_{\beta(t)})\setminus (\l_{p_0p_1}\cup \ldots \cup \l_{p_0p_4}\cup \gamma_0)$.
\end{prop}
\begin{rem}\label{rem:twodim}
The symmetric group $\mathrm{S}_4$ of four letters  acts on $p_1, \ldots, p_4$ and then acts on $\PP^2$ faithfully.
Note that $(13)(24)$ preserves $\l_{r_0s_0}$ and maps the point $p_0(t)=(t:1:1+t)$ to $p_0(-t)=(-t:1:1-t)$. Also it preserves $\c_{\alpha(t)}$.
Observe that $\beta(-t)=\alpha(t)$. So $(13)(24)$ acts on
$$\{(p_0, p)|p_0\in \l_{r_0s_0}\setminus\{r_0,s_0,q_1,q_2\},
                 p\in (\c_{\alpha(t)}\cup \c_{\beta(t)})\setminus(\l_{p_0p_1}\cup \ldots \cup \l_{p_0p_4}\cup \gamma_0)\}.$$
and the quotient is $2$-dimensional and irreducible.
We conclude that isomorphism classes of the sets $\P$ of six points satisfying conditions (I)-(IV) are parametrized by a $2$-dimensional irreducible variety.
\end{rem}

\section{Certain rational surfaces with eight nodal curves and two elliptic fibrations}\label{ratellsur}
Fix a set $\P$ of six points of $\CC\PP^2$ satisfying the conditions (I)-(IV) in Section \ref{confpoints}.

Let $\rho \colon W \rightarrow \CC\mathbb{P}^2$ be the blowup of $\CC\mathbb{P}^2$ at the ten points $p_0, p_1, \ldots, p_4, p, p_1', \ldots, p_4'$.
Denote by $L$ the pullback of a line by $\rho$,
by $E_k$ the \textbf{total transform} of $p_k$, by $E_k'$ the exceptional divisor of $p_k'$ ($k=1, 2, 3, 4$),
by $E_0$ the exceptional divisor of $p_0$ and by $E$ the exceptional divisor of $p$.

\begin{enumerate}[\upshape (1)]
\item The Picard group of $W$ is
$$\mathrm{Pic}(W)=\mathbb{Z}L\oplus\mathbb{Z}E_0\oplus
\oplus_{k=1}^4(\mathbb{Z}E_k\oplus\mathbb{Z}E_k')\oplus\mathbb{Z}E,$$
 and thus $\rho(W)=11$.
 Moreover,
 \[
 -K_W= 3L-E_0-\sum_{k=1}^4 (E_k+E_k')-E \textrm{ and } K_W^2=-1.
 \]
 \item  $W$ contains the following nodal curves:
        \begin{itemize}
            \item for $k=1, \ldots, 4$, the strict transform $C_k$ of $\l_{p_0p_k}$:
                    $$C_k \equiv L-E_0-E_k-E_k';$$
            \item for $k=1, \ldots, 4$, the strict transform $C_k'$ of the $(-1)$-curve corresponding to $p_k$:
                        $$C_k' \equiv E_k-E_k';$$
            \item  the strict transforms $\Gamma_1,\Gamma_2$ of $\gamma_1$ and $\gamma_2$ (see
                         the condition~(II)):
                        \begin{align*}
                            \Gamma_1 &\equiv 2L-E_1-E_2-E_3-E_3'-E_4-E_4',\\
                            \Gamma_2 &\equiv 2L-E_1-E_1'-E_2-E_2'-E_3-E_4;
                        \end{align*}
            \item  the strict transforms $\Lambda_1,\Lambda_2$ of $\lambda_1$ and $\lambda_2$
                        (see the condition~(IV)):
                    \begin{align*}
                        \Lambda_1 &\equiv 3L-E_0-E_1-E_2-E_2'-E_3-E_4-E_4'-2E,\\
                        \Lambda_2 &\equiv 3L-E_0-E_1-E_1'-E_2-E_3-E_3'-E_4-2E.
                    \end{align*}
        \end{itemize}

        		Observe that the eight nodal curves $C_1,C_1',\dots,C_4,C_4'$ are disjoint and the curves
		\begin{align*}
		2\Gamma_1+C_1+C_1'+C_2+C_2',\ 2\Gamma_2+C_3+C_3'+C_4+C_4',\\ 2\Lambda_1+C_1+C_1'+C_3+C_3',\ 2\Lambda_2+C_2+C_2'+C_4+C_4'
		\end{align*}
        are of type $\tilde{D}_4$.

       \item $W$ also contains the following irreducible smooth elliptic curve:
     		\begin{itemize}
        		\item The strict transform $\Gamma_0$ of $\gamma_0$
                   \begin{align*}
                        \Gamma_0 &\equiv 3L-E_0-E_1-E_1'-E_2-E_2'-E_3-E_3'-E_4-E_4' \\
                                           & \equiv -K_W+E \textrm{ (see Lemma \ref{lem:gamma0} and the condition (III)).}
                    \end{align*}
                \end{itemize}
\end{enumerate}

\begin{lem}\label{lem:rational}
Let $F:=L-E_0$. Then $|F|$ defines a rational fibration $f \colon W \rightarrow \PP^1$.
 The singular fibers of $f$ are as follows:
   \begin{enumerate}[\upshape (a)]
        \item $E+\B_3$ with $E\B_3=1$, where $\B_3$ is a $(-1)$-curve which is the strict transform of $\l_{p_0p}$, and $\B_3 \equiv L-E_0-E$;
        \item $C_k+2E_k'+C_k'$ with $C_kE_k'=C_k'E_k'=1$ for $k=1,\ldots, 4$.
    \end{enumerate}
\end{lem}
Proof. Note that $|F|$ is the pullback of the pencil of lines passing through $p_0$ by $\rho$. \qed.

\begin{prop}\label{prop:elliptic1}
Let $\Gamma:=-2K_W+2E$. Then $|\Gamma|$ defines an elliptic fibration $h_1 \colon W \rightarrow \PP^1$
with all the singular fibers as follows:
\begin{enumerate}[\upshape (a)]
    \item two reducible fibers of type $I_0^*$:
             $$2\Gamma_1+C_1+C_1'+C_2+C_2',\ 2\Gamma_2+C_3+C_3'+C_4+C_4';$$
    \item the double fiber $2\Gamma_0$, where $\Gamma_0$ is in {\rm (3)};
    \item a reducible fiber $\B_1+E$,
           where $\B_1$ is an irreducible smooth elliptic curve such that
           $$\B_1\equiv -2K_W+E,\ \B_1E=1.$$
\end{enumerate}
Moreover, $\B_3$ intersects $\Gamma_0, \Gamma_1$ and $\Gamma_2$ transversely.
\end{prop}
\begin{proof}
Note that $\Gamma^2=0$ and $K_W\Gamma=0$. The adjunction formula yields $p_a(\Gamma)=1$.
From (2) and (3), one sees immediately that
$$2\Gamma_1+C_1+C_1'+C_2+C_2',\ 2\Gamma_2+C_3+C_3'+C_4+C_4',\ 2\Gamma_0$$
are connected, linearly equivalent to $\Gamma$ and that they are disjoint.\
Hence $|\Gamma|$ defines an elliptic fibration $h_1\colon W\rightarrow \mathbb{P}^1$, and (a) and (b) are clear.

Since $E\Gamma=E(-2K_W+2E)=0$, $E$ is contained in some singular fiber different from those in (a) and (b). We may write it
as $\B_1+E$ with $\B_1\ge 0$.

Note that $E$ is the unique $(-1)$-curve contained in the fiber of $h_1$.\ Denote by $\rho_1\colon W\rightarrow W_1$ the contraction of $E$ and by $h_1'\colon W_1 \rightarrow \mathbb{P}^1$ the elliptic fibration induced by $h_1$. Let $2\Gamma_0', m_1\Gamma_1', \dots, m_r\Gamma_r'$ be all the multiple fibers of $h_1'$, where $\Gamma_0':={\rho_1}(\Gamma_0)$. By applying the canonical bundle formula for $h_1'$ (see \cite[V (12.3) Corollary]{WBCPAV84}), we have
\begin{align*}
K_{W_1} & \equiv h_1'^*(\mathcal{O}_{\mathbb{P}^1}(-1))+\Gamma_0'+\sum_{j=1}^r(m_j-1)\Gamma_j'\\
            & \equiv -\Gamma'+\Gamma_0'+\sum_{j=1}^r(m_j-1)\Gamma_j',
\end{align*}
where $\Gamma':={\rho_1}(\textbf{a general fiber of $h_1$})$. Since $\Gamma'\equiv -2K_{W_1}$ and $\Gamma_0'\equiv -K_{W_1}$ we obtain
\[
\sum_{j=1}^r(m_j-1)\Gamma_j'\equiv 0.
\]
It means that $2\Gamma_0'$ is the only multiple fiber of $h_1'$
and so is $2\Gamma_0$ the only one of $h_1$.

We also have
$e(W_1)=\sum_{x \in \PP^1}e(\Gamma_x')$, and $e(\Gamma_x')\ge 0$ and $e(\Gamma_x')=0$ if and only if the support of $\Gamma_x'$ is an irreducible smooth elliptic curve (see \cite[III (11.4) Proposition]{WBCPAV84} and \cite[Lemma VI.4]{AB78}), where $\Gamma_x'$ is a fiber of $h_1'$ over a point $x \in \mathbb{P}^1$.

Note that $e(W_1)=12\chi(\O_{W_1})-K_{W_1}^2=12$, $e(\textbf{a fiber of type $I_0^*$})=6$ and $e(2\Gamma_0')=0$. Thus $e(\Gamma_{y}')=0$ for any fiber $\Gamma_{y}'$ of $h_1'$ different from the two fibers of type $I_0^*$ and $2\Gamma_0'$. Together with the fact that $2\Gamma_0'$ is the only multiple fiber of $h_1'$, we conclude that $\Gamma_{y}'$ is smooth and so is particularly $\B_1'$, where $\B_1':={\rho_1}(\B_1+E)$. Hence ${\rho_1}^*(\B_1')=\B_1+E$ and $\B_1$ is irreducible smooth. The singular fibers of $h_1$ are as required.

Consider $h:=h_1|_{\B_3} \colon \B_3 \rightarrow \PP^1$.
Then $\deg h=\B_3(-2K_W+2E)=4$ and thus $\deg R_{h}=6$, where $R_h$ is the ramification divisor of $h$.
Observe that $\Gamma_j|_{\B_3} \le R_{h}$ and $\deg(\Gamma_j|_{\B_3})=2$ for $j=0,1,2$.
Hence $R_{h}=\sum_{j=0}^2\Gamma_j|_{\B_3}$ and $\B_3$ intersects $\Gamma_j$ transversely.
\end{proof}

\begin{cor}\label{cor:-2KW}
$H^0(W, \O_W(-2K_W))=0$.
\end{cor}
\begin{proof}
It follows from $\B_1\equiv -2K_W+E$, $\B_1^2=-1$ and that $\B_1$ is irreducible.
\end{proof}

\begin{prop}\label{prop:elliptic2}
Let $\Lambda:=-2K_W+2\B_3$. Then $|\Lambda|$ defines an elliptic fibration $h_2 \colon W \rightarrow \PP^1$
with all the singular fibers as follows:
\begin{enumerate}[\upshape (a)]
    \item two reducible fibers of type $I_0^*$:
             $$2\Lambda_1+C_1+C_1'+C_3+C_3',\ 2\Lambda_2+C_2+C_2'+C_4+C_4';$$
    \item the double fiber $2\Lambda_0$, where $\Lambda_0 \equiv -K_W+\B_3$ is an irreducible smooth elliptic curve;
    \item a reducible fiber $\B_2+\B_3$,
           where $\B_2$ is an irreducible smooth elliptic curve such that
           $$\B_2\equiv -2K_W+\B_3,\ \B_2\B_3=1;$$
\end{enumerate}
Moreover, $E$ intersects $\Lambda_0, \Lambda_1$ and $\Lambda_2$ transversely.
\end{prop}
\begin{proof}
Since $\Lambda^2=0$ and $K_W\Lambda=0$, the adjunction formula gives $p_a(\Gamma)=1$.
One sees immediately that the two divisors
$$2\Lambda_1+C_1+C_1'+C_3+C_3',\ 2\Lambda_2+C_2+C_2'+C_4+C_4'$$
are connected, pairwise disjoint and contained in the linear system $|\Lambda|$.
Hence $|\Lambda|$ induces an elliptic fibration $h_2\colon W\rightarrow \mathbb{P}^1$.
The assertion (a) is clear.

Since $\B_3\Lambda=\B_3(-2K_W+2\B_3)=0$, $\B_3$ is contained in some singular fiber of $h_2$, which we may write it
as $\B_2+\B_3$ with $\B_2 \ge 0$. Then $\B_2\equiv -2K_W+\B_3, \B_2\B_3=1$ and $\B_2 \not\ge \B_3$ by \lemref{cor:-2KW}. Then the main statement is that $\overline{B_2}$ is irreducible and smooth.

Note that $\B_3$ is the unique $(-1)$-curve contained in the fiber of $h_2$. Denote by $\rho_2\colon W\rightarrow W_2$ the contraction of $\B_3$ and by $h_2'\colon W_2 \rightarrow \mathbb{P}^1$ the elliptic fibration induced by $h_2$. Let $n_0\Lambda_0', n_1\Lambda_1', \dots, n_s\Lambda_s'$ be all the multiple fibers of $h_2'$. By applying the canonical bundle formula for $h_2'$ (see \cite[V (12.3) Corollary]{WBCPAV84}), we have
\[
K_{W_2}  \equiv h_2'^*(\mathcal{O}_{\mathbb{P}^1}(-1))+\sum_{j=0}^s(n_j-1)\Lambda_j'\equiv -\Lambda'+\sum_{j=0}^s(n_j-1)\Lambda_j',
\]
where $\Lambda':={\rho_2}(\textbf{a general fiber of $h_2$})$. Since $\Lambda'\equiv -2K_{W_2}$ we get
\[
\sum_{j=0}^s\frac{n_j-1}{n_j}=\frac{1}{2}.
\]
It implies $s=0$ and $n_0=2$. So $h_2$ has the unique multiple fiber which is a double fiber,
because so does $h_2'$. Denote it by $2\Lambda_0$.
Then $\Lambda_0\equiv -K_W+\B_3$ and $\Lambda_0 \not\ge \B_3$ by \lemref{cor:-2KW}.

We also have
$e(W_2)=\sum_{x \in \PP^1}e(\Lambda_x')$, and $e(\Lambda_x')\ge 0$ and $e(\Lambda_x')=0$ if and only if the support of $\Lambda_x'$ is an irreducible smooth elliptic curve (see \cite[III (11.4) Proposition]{WBCPAV84} and \cite[Lemma VI.4]{AB78}), where $\Lambda_x'$ is a fiber of $h_2'$ over a point $x \in \mathbb{P}^1$.

Note that $e(W_2)=12\chi(\O_{W_2})-K_{W_2}^2=12$, $e(\textbf{a fiber of type $I_0^*$})=6$ and $e(2\Lambda_0')=0$. Thus $e(\Lambda_{y}')=0$ for any fiber $\Lambda_{y}'$ of $h_2'$ different from the two fibers of type $I_0^*$ and $2\Lambda_0'$. Together with the fact that $2\Lambda_0'$ is the only multiple fiber of $h_2'$, we conclude that $\Gamma_{y}'$ is smooth and so is particularly $\B_2'$, where $\B_2':={\rho_2}(\B_2+\B_3)$. Hence ${\rho_2}^*(\B_2')=\B_2+\B_3$ and $\B_2$ is irreducible smooth. The singular fibers of $h_2$ are as required.

Consider $h':=h_2|_{E} \colon E \rightarrow \PP^1$.
Then $\deg h'=E(-2K_W+2\B_3)=4$ and thus $\deg R_{h'}=6$, where $R_{h'}$ is the ramification divisor of $h'$.
Observe that $\Lambda_j|_{E} \le R_{h'}$ and $\deg(\Lambda_j|_{E})=2$ for $j=0,1,2$.
Hence $R_{h'}=\sum_{j=0}^2\Lambda_j|_{E}$ and $E$ intersects $\Lambda_j$ transversely.
\end{proof}

\begin{rem}\label{rem:notcubic}
From Proposition \ref{prop:elliptic2}, we see that
\begin{itemize}
    \item[i)] $\Lambda_0 \equiv -K_W+\B_3 \equiv 4L-2E_0-E_1-E_1'-E_2-E_2'-E_3-E_3'-E_4-E_4'-2E.$
Thus $\P \subseteq \CC\PP^2$ satisfies the following property:
	\begin{itemize}
  	  \item[\upshape (IV')] there is an irreducible quartic $\lambda_0$ passing through the eight points\\
                    $p_1, p_1', p_2, p_2', p_3, p_3', p_4, p_4'$ and having double points at $p_0, p$.
	\end{itemize}
    \item[ii)] The singular point $p$ of $\lambda_1$ (resp. $\lambda_2$)
           is indeed a node.
   \item[iii)] $\B_1\equiv 6L-2E_0-2\sum_{k=1}^4(E_k+E_k')-E$.\ Thus $\P$ satisfies that there is an irreducible sextic $\rho(\B_1)$ passing through the point $p$, having the double point at $p_0$ and $[2,2]$-points at $p_k$ for $k=1,2,3,4$.
	 \item[iv)] $\B_2 \equiv 7L-3E_0-2\sum_{k=1}^4(E_k+E_k')-3E$.\ Thus $\P$ satisfies that there is an irreducible septic $\rho(\B_2)$ having  triple points at $p_0, p$ and $[2,2]$-points at $p_k$ for $k=1,2,3,4$.
\end{itemize}
It is much harder to find $\lambda_0$, $\rho(\B_1)$ or $\rho(\B_2)$ directly than to find the cubics $\gamma_0, \lambda_1, \lambda_2$ in the condition (IV) and Lemma \ref{lem:gamma0}. 
\end{rem}

 \begin{lem}\label{lem:snc}
  The divisor $\B_1+\B_2+\B_3$ is simple normal crossing with $\B_1\B_2=1, \B_1\B_3=3$ and $\B_2\B_3=1$.
 \end{lem}
 \begin{proof}
 The intersection numbers follow by a direct calculation.
 It is sufficient to show that $\B_1$ intersects with $\B_3$ transversely, and $\B_1,\B_2,\B_3$ have no common point.

 First consider $f':=f|_{\B_1}\colon \B_1 \rightarrow \mathbb{P}^1$. Then $\deg f'=4$ and so $\deg R_{f'}=8$ by Hurwitz's formula, where $R_{f'}$ is the ramification divisor of $f'$. By Lemma \ref{lem:rational},
 $R_{f'}\ge \sum_{k=1}^4 E_k'|_{\B_1}$ and indeed the equality holds because both sides have the same degree.
It means $\B_1$  intersects the fiber $E+\B_3$ of $f$ transversely.

Now let
\[
\Psi:=\B_1+\B_3 \equiv \B_2+E
\]
(see Propositions \ref{prop:elliptic1} and \ref{prop:elliptic2}).
Since $H^1(W,\O_W(E))=0$ and $\Psi\B_2=2$ the short exact sequence \[
0\rightarrow \mathcal{O}_W(E) \rightarrow \mathcal{O}_W(\B_2+E) \rightarrow \mathcal{O}_{\B_2}(\B_2+E)\rightarrow 0
\]
and Riemann-Roch theorem on $\B_2$ show that the trace of $|\Psi|$ on $\B_2$ is complete and base point free.
Denote by $g'\colon \B_2\rightarrow \mathbb{P}^1$ the double cover induced by $|\Psi|_{\B_2}|$. Then $\deg g'=2$ and $\deg R_{g'}=4$, where $R_{g'}$ is the ramification divisor of $g'$.

Since $F\equiv E+\B_3$
\begin{equation}\label{B1+B3Dk}
\Psi=\B_1+\B_3  \equiv -2K_W+F\equiv 2(-K_W+E_k')+C_k+C_k'
\end{equation}
for any $k=1,2,3,4$. By Riemann-Roch theorem
\[
h^0(W,\mathcal{O}_W(-K_W+E_k'))\ge 1
\]
because $(-K_W+E_k')^2=0$ and $K_W(-K_W+E_k')=0$. Thus
\[
-K_W+E_k'\equiv \Psi_k \textrm{ for some effective divisor $\Psi_k$ on $W$.}
\]
Hence $ R_{g'}\ge \sum_{k=1}^4 \Psi_k|_{\B_2}$ and indeed the equality holds because both sides have the same degree.
It means that $(\B_1+\B_3)|_{\B_2}$ consists of two distinct points.
\end{proof}

\begin{lem}\label{lem:B2F}
The divisor $\B_2+F$ is nef with $(\B_2+F)^2=7$. If $C$ is an irreducible curve such that $(\B_2+F)C=0$,
then $C$ is one of the nodal curves $C_1, C_1', \ldots, C_4, C_4'$.
\end{lem}
\proof Assume that $C$ is an irreducible curve such that $(\B_2+F)C \le 0$.
Since $(\B_2+F)\B_2=3$ and $F$ is nef, we conclude that $C \not= \B_2$ and $\B_2C=FC=0$.
This means $C$ is contained in singular fibers of $f$ and $(-2K_W+\B_3)C=0$.
The conclusion follows from \lemref{lem:rational}. \qed

\begin{lem}\label{lem:bir}
The linear system $|\B_2+F|$ defines a birational morphism $\phi\colon W\rightarrow \mathbb{P}^4$ which contracts exactly eight $(-2)$-curves $C_k,C_k'$ for $k=1,2,3,4$.
\end{lem}
\begin{proof}
Note that $H^1(W,\O_W(F))=0$. Consider the exact sequence
\[
0\rightarrow H^0(W,\O_W(F)) \rightarrow H^0(W,\O_W(\B_2+F)) \rightarrow H^0(\B_2,\O_{\B_2}(\B_2+F))\rightarrow 0.
\]
Since $\deg(\O_{\B_2}(\B_2+F))=3\ge2g(\B_2)=2$, $|\O_{\B_2}(\B_2+F)|$ is base point free and of dimension $2$. And since $|F|$ is a base-point-free pencil, $|\B_2+F|$ is base point free and of dimension $4$. Because $(\B_2+F)^2$ is the prime number $7$, the morphism $\phi \colon W \rightarrow \mathbb{P}^4$ is birational.

By Lemma \ref{lem:B2F}, $\phi$ only contracts the nodal curves $C_k,C_k'$ for $k=1,2,3,4$.
\end{proof}

\section{Construction of surfaces of general type}\label{construction}
In this section we construct a $2$-dimensional family (see Remark \ref{rem:twodim}) of minimal surfaces of general type with $p_g=0$ and $K^2=7$, having commuting involutions. Those surfaces belong to the case $(c)$ of \cite[Theorem 1.1]{YC15} (see Proposition \ref{prop:intermediate}).

We define
\begin{equation}\label{Deltai}
\begin{array}{rl}
\Delta_{1}\!\!\!\!&:=\B_1+C_{1}+C_{1}'+C_{2}+C_{2}',\\
\Delta_{2}\!\!\!\!&:=\B_2+C_{3}+C_{3}',\\
\Delta_{3}\!\!\!\!&:=\B_3+C_{4}+C_{4}'.
\end{array}
\end{equation}
Then
\begin{align*}
\Delta_{1}&\equiv -2K_W+2L-2E_0-2E_1'-2E_2'+E,\\
\Delta_{2}&\equiv -2K_W+2L-2E_0-2E_3'-E,\\
\Delta_{3}&\equiv 2L-2E_0-2E_4'-E.
\end{align*}
Let $\Delta:=\Delta_1+\Delta_2+\Delta_3$. Since $\B_1, \B_2, \B_3$ are disjoint from the
nodal curves $C_1, \ldots, C_4'$, $\Delta$ is simple normal crossing by \lemref{lem:snc}.

Note that $\mathrm{Pic}(W)$ is torsion-free. For $i=1,2,3$, there exists a unique $\L_{i} \in \mathrm{Pic}(W)$ such that
\begin{align*}
2\L_{i}\equiv \Delta_{i+1}+\Delta_{i+2},\ \ \L_{i}+\Delta_{i} \equiv \L_{i+1}+\L_{i+2},
\end{align*}
where
\begin{equation}
	\begin{array}{rl}\label{Li}
	\L_{1}\!\!\!\!\! &:=-K_W+2L-2E_0-E_3'-E_4'-E,\\
	\L_{2}\!\!\!\!\! &:=-K_W+2L-2E_0-E_1'-E_2'-E_4',\\
	\L_{3}\!\!\!\!\! &:=-2K_W+2L-2E_0-E_1'-E_2'-E_3',
	\end{array}
\end{equation}
and the index $i\in \{1,2,3\}$ is considered as modulo $3$.

There is a smooth bidouble cover $\overline{\pi} \colon V \rightarrow W$ branched along $\Delta$ (see \cite[Section 1]{FC84}, \cite[Theorem 2]{FC99} and  \cite{Pardini91}).
Note that $\overline{\pi}^{-1}(C_k)$ (repectively $\overline{\pi}^{-1}(C_k'))$
are disjoint union of two $(-1)$-curves for $k=1,2,3,4$.
Let $\epsilon \colon V \rightarrow S$ be the blowdown of these sixteen $(-1)$-curves. We have a commutative diagram:
\begin{displaymath}
\xymatrix{
  V  \ar[r]^{\epsilon}  \ar[d]_{\bpi} & S  \ar[d]^{\pi}\\
  W \ar[r]^{\eta}                                 & \Sigma }
\end{displaymath}
where $\eta$ is a contraction of eight $(-2)$-curves $C_k, C_k'$ for $k=1,2,3,4$, and $\pi$ is a bidouble cover $\pi \colon S \rightarrow \Sigma$ branched along $\eta(\Delta)$ and the eight nodes of $\Sigma$. Remark the notation $W$ is different from ones of \cite{CCM07, CMLP08,YLYS14}.
\begin{thm}\label{thm:mainthm}
The surface $S$ is a minimal surface of general type with $p_g(S)=0$ and $K^2_S=7$. Moreover, the canonical divisor $K_S$ is ample.
\end{thm}
\begin{proof}
Note
\begin{align*}
\epsilon^*(2K_S)+\bpi^*\left(\sum_{k=1}^4(C_k+C_k')\right) & \equiv 2K_V= \bpi^*(2K_W+\Delta)  \textrm{ (see  \cite[Section 2]{FC99} and \cite{Pardini91})} \\
                  						 	    &\equiv \bpi^*(D)+\bpi^*\left(\sum_{k=1}^4(C_k+C_k')\right),
\end{align*}
where
\begin{align}\label{D}
D:=2K_W+\B_1+\B_2+\B_3\equiv \B_2+F.
\end{align}
It implies
\begin{align}\label{2KBF}
\epsilon^*(2K_S)\equiv \bpi^*(D).
\end{align}
By Lemma \ref{lem:B2F} $K_S^2=\frac{1}{4}(2K_S)^2=\frac{1}{4}4D^2=7$.

Since $p_g(W)=0$, by \cite[Section 2]{FC99}, we have
\begin{align*}
    p_g(S)=p_g(V)=p_g(W)+\sum_{i=1}^3h^0(W, \O_W(K_W+\L_i))=\sum_{i=1}^3h^0(W, \O_W(K_W+\L_i)).
\end{align*}
It is clear that
\begin{align*}
H^0(W,\mathcal{O}_W(K_W+\L_1))=H^0(W,\mathcal{O}_W(2L-2E_0-E_3'-E_4'-E))=0,\\
H^0(W,\mathcal{O}_W(K_W+\L_2))=H^0(W,\mathcal{O}_W(2L-2E_0-E_1'-E_2'-E_4'))=0.
\end{align*}

Assume $\Phi \in |K_W+\L_3|$.
For $j=1, 2, 3$, since $\Phi C_j=\Phi C_j'=-1$, $\Phi \ge C_j+C_j'$. Then
\[
\Phi-\sum_{j=1}^3(C_j+C_j') \ge 0.
\]
 Since $[\Phi-\sum_{j=1}^3(C_j+C_j')]\Gamma_1=-1$, $\Phi \ge \sum_{j=1}^3(C_j+C_j')+\Gamma_1$.
 But $\Phi-(\sum_{j=1}^3(C_j+C_j')+\Gamma_1) \equiv E_3'-E$, which is impossible.

Hence $|K_W+\L_3|=\emptyset$ and thus $p_g(S)=0$.
By Lemma \ref{lem:B2F} and (\ref{2KBF}), $K_S$ is ample.
\end{proof}

\section{The intermediate double covers and the bicanonical map}\label{intermediate}
Let $S$ be the surface constructed in Section \ref{construction}, $G$ be the Galois group $\mathbb{Z}_2\times \mathbb{Z}_2=\{\mathrm{Id}_S,g_1,g_2,g_3\}$ of the bidouble cover of Section \ref{construction}, $R_i$ be the divisorial fixed part of $g_i$ and $k_{i}$ be the number of the isolated fixed points of $g_i$ for each $i=1,2,3$. Denote by $S/g_i$ the quotient of $S$ by $g_i$.
\begin{prop}\label{prop:intermediate}
Let $S$ be the surface constructed in Section \ref{construction} and keep the notation as above. Then $R_i$ is irreducible and $R_i^2=-1$ for $i=1,2,3$, and
\begin{itemize}
	\item[i)] $K_SR_1=5$, $g(R_1)=3$ and $k_{1}=9$. Moreover $S/g_1$ is birational to an Enriques surface and $K_{S/g_1}^2=-2$.
	\item[ii)] $K_SR_2=3$, $g(R_2)=2$ and $k_{2}=7$. Moreover $S/g_2$ is birational to a minimal properly elliptic surface with $K_{S/g_2}^2=0$ and $\kappa(S/g_2)=1$.
	\item[iii)] $K_SR_3=1$, $g(R_3)=1$ and $k_{3}=5$. Moreover the minimal resolution $T$ of  $S/g_3$ is a minimal surface of general type with $p_g(T)=0$ and $K_{T}^2=2$ (i.e. a numerical Campedelli surface).
\end{itemize}
\end{prop}
\begin{proof}
The intersection numbers $R_i^2$ and $K_SR_i$ follow from
\[R_i^2=\B_i^2, K_SR_i=D\B_i,\]
(see \cite[Lemma 2.4 (b)]{YC15}, and (\ref{D}) for $D$). Since $R_i$ is isomorphic to $\B_i$,
$R_i$ is irreducible because so is $\B_i$. Then $g(R_i)$ and $k_i$ follow from
the adjunction formula and \cite[Lemma 4.2]{DMP02}.

We see that $(S, G)$ is in the case (c) of \cite[Theorem~2.9]{YC15} and
statement (iii) is contained in  \cite[Proposition~2.10]{YC15}. It remains to calculate
$\kappa(S/g_1)$ and $\kappa(S/g_2)$. Note that $S/g_1$ and $S/g_2$ have only nodes as singularities.
\medskip

i) Let $\bpi_1\colon V_1\rightarrow W$ be the intermediate singular double cover of $\bpi\colon V\rightarrow W$ associated to the data $2\L_1\equiv \Delta_2+\Delta_3$. Note that $V_1$ contains four disjoint $(-1)$-curves $\bpi_1^{-1}(C_k), \bpi_1^{-1}(C_k')$ for $k=3,4$ and eight nodal curves
$\bpi_1^{-1}(C_j), \bpi_1^{-1}(C_j')$ for $j=1,2$.  Contract all these curves, and we obtain $S/g_1$.  Therefore $K_{S/g_1}^2=K_{V_1}^2+4$ and $\kappa(S/g_1)=\kappa(V_1)$. On the other hand
\begin{displaymath}
	\begin{array}{rl}
		2K_{V_1} \!\!\!\!&\equiv \bpi_1^*(2K_W+2\L_1) \equiv \bpi_1^*(2K_W+\B_2+\B_3+\sum_{k=3,4}(C_k+C_k'))  \textrm{ (see  (\ref{Deltai}) and (\ref{Li}))}\\
						       & \equiv \bpi_1^*(2\B_3)+\bpi_1^*(\sum_{k=3,4}(C_k+C_k')) \textrm{ (see Proposition \ref{prop:elliptic2} (c))}.
	\end{array}
\end{displaymath}
It follows that $\kappa(V_1)=\dim |\B_3|=0$ and $K_{V_1}^2=-6$ and thus
$\kappa(S/g_1)=0$ and $K_{S/g_1}^2=-2$.
Since $p_g(S/g_1)=0$ and $q(S/g_1)=0$, $S/g_1$ is birational to an Enriques surface.
\medskip

ii) Let $\bpi_2\colon V_2\rightarrow W$ be the intermediate singular double cover of $\bpi\colon V\rightarrow W$ associated to the data $2\L_2\equiv \Delta_1+\Delta_3$. Note that $V_2$ contains six disjoint $(-1)$-curves $\bpi_2^{-1}(C_k), \bpi_2^{-1}(C_k')$ for $k=1,2,4$  and four nodal curves
$\bpi_2^{-1}(C_3), \bpi_2^{-1}(C_3')$. Contract all these curves, and we obtain $S/g_2$.  Therefore $K_{S/g_2}^2=K_{V_2}^2+6$ and $\kappa(S/g_2)=\kappa(V_2)$. On the other hand
\begin{displaymath}
	\begin{array}{rl}
		2K_{V_2} \!\!\!\!& \equiv \bpi_2^*(2K_W+2\L_2) \equiv \bpi_2^*(2K_W+\B_1+\B_3+\sum_{k=1,2,4}(C_k+C_k'))  \textrm{ (see  (\ref{Deltai}) and (\ref{Li}))}\\
						       & \equiv \bpi_2^*(E+\B_3)+\bpi_2^*(\sum_{k=1,2,4}(C_k+C_k'))\equiv \bpi_2^*(F)+\bpi_2^*(\sum_{k=1,2,4}(C_k+C_k'))\\
					&  \textrm{ (see Lemma \ref{lem:rational} (a) and Proposition \ref{prop:elliptic1} (c))}.
	\end{array}
\end{displaymath}
It follows that $\kappa(V_2)=\dim |F|=1$ and $K_{V_2}^2=-6$ and thus
$\kappa(S/g_2)=1$ and $K_{S/g_2}^2=0$.
\end{proof}

\begin{rem}
By Proposition \ref{prop:intermediate} the surfaces $S$ constructed in Section \ref{construction} are contained in the case of \cite[Theorem 1.1 (c)]{YC15}.
\end{rem}

\begin{rem}\label{branchEnriques}
Lee and Shin \cite{YLYS14} gave possible branch divisors of involutions on minimal surfaces of general type with $p_g=0$ and $K^2=7$. Proposition \ref{prop:intermediate} provides three examples for the possible branch divisors in the list of \cite{YLYS14}:
\begin{itemize}
	\item[(1)] $k=9$, $K_W^2=-2$, $W$ is birational to an Enriques surface and $B_0=\substack{\Gamma_{0}\\(3,-2)}$.
	\item[(2)] $k=7$, $K_W^2=0$, $W$ is a minimal properly elliptic surface and $B_0=\substack{\Gamma_{0}\\(2,-2)}$.
	\item[(3)] $k=5$, $K_W^2=2$, $W$ is a minimal surface of general type with $p_g(W)=0$ and $K_W^2=2$ (i.e. a numerical Campedelli surface), and $B_0=\substack{\Gamma_{0}\\(1,-2)}$.
\end{itemize}
We note the three branch divisors of the surface $S$ constructed in Section \ref{construction} are different from ones of an Inoue's surface (see \cite[Section 5]{YLYS14}) and the surface constructed by the first named author (see \cite[Remark 5.1]{YC13}).
\end{rem}

The group $G$ acts on $H^0(S,\O_S(2K_S))$. Let $G^*=\{1,\chi_1,\chi_2,\chi_3\}$ be the character group of $G$, where $\chi_i$ is orthogonal to $g_i$ for $i=1,2,3$. Then
\[
H^0(S,\O_S(2K_S))=H^0(S,\O_S(2K_S))^{\mathrm{inv}}\bigoplus \bigoplus_{\chi_i\in G^*\setminus \{1\}} H^0(S,\O_S(2K_S))^{\chi_i}
\]
\begin{cor}\label{cor:eigenspace}
For the surface $S$ constructed in Section \ref{construction}
\begin{displaymath}
	\begin{array}{l}
		h^0(S,\O_S(2K_S))^{\mathrm{inv}}=5, \\ h^0(S,\O_S(2K_S))^{\chi_1}=2,\ h^0(S,\O_S(2K_S))^{\chi_2}=1,\ h^0(S,\O_S(2K_S))^{\chi_3}=0.
	\end{array}
\end{displaymath}
\end{cor}
\begin{proof}
By Proposition \ref{prop:intermediate} $(K_SR_1,K_SR_2,K_SR_3)=(5,3,1)$. We apply \cite[Proposition 2.2]{YC15}.
\end{proof}

\begin{cor}\label{cor:notcomposed}
The bicanonical morphism $\varphi\colon S\rightarrow \mathbb{P}^7$ is not composed with $g_i$ for any $i=1,2,3$.
\end{cor}
\begin{cor}\label{cor:birational}
Let $S$ be the surface constructed in Section \ref{construction}. Then the bicanonical morphism $\varphi\colon S\rightarrow \mathbb{P}^7$ is birational.
\end{cor}
\begin{proof}
The linear system $|\phi(B_2+F)|$ defines an embedding  $\Sigma \rightarrow \mathbb{P}^4$ by Lemma \ref{lem:bir}. Moreover $2K_S\equiv \pi^*(\phi(B_2+F))$ (see (\ref{2KBF})) and $\pi\colon S\rightarrow \Sigma$ is a $G$-Galois covering of $\Sigma$. Thus $\varphi$ is birational if and only if $\varphi$ is not composed with all $g_1,g_2$ and $g_3$.
By Corollary \ref{cor:notcomposed} $\varphi$ is birational.
\end{proof}

\section{Appendix}\label{pfcp}

Before the proofs of \lemref{lem:p0}, \lemref{lem:gamma0} and \propref{prop:p},
we make the following remark.
\begin{rem}\label{rem:upshot}
 Assume that $(\gamma_0, o)$ is an elliptic curve with its group structure, $p_0, r_0, s_0$ are nontrivial $2$-torsion points of $\gamma_0$ and $p_1, \ldots, p_4$ are halves of $p_0$. We may assume that the sum of $p_1$ and $p_2$ is $o$, i.e. $p_1+p_2\equiv 2o$. Then $p_3+p_4\equiv 2o$. By identifying
$\gamma_0$ with its image under the embedding into $\PP^2$ by $|3o|$, one
sees easily that $p_0, p_1, \ldots, p_4$ satisfy conditions (I)-(III).
For example, let $\gamma_2$ be the conic passing through five points $p_1, p_2, p_3, p_4, p_1'$.
Then
$$\gamma_2|_{\gamma_0}-(2p_1+p_2+p_3+p_4)\equiv 6o-(4o+p_1)\equiv p_2.$$
This means $\gamma_2$ contains $p_2'$. However we fail to find $p$ or to verify the condition (IV) in such a geometrical way.
\end{rem}

We use the same notation in Section~2.

\begin{proof}[Proof of \lemref{lem:p0}]

A smooth conic $\c_\alpha$ passing through  $p_1, p_2, p_3, p_4$ has the defining equation
       $$-(1+\alpha)x_1x_2+x_1x_3+\alpha x_2x_3=0.$$
       where $\alpha \not =0, -1$.
       The tangent lines of $\c_\alpha$ to the points $p_1, p_2, p_3, p_4$ have the following defining equations:
\begin{displaymath}
       \begin{array}{lll}
       T_{p_1}\c_\alpha: -(1+\alpha)x_2+x_3=0,&&
       T_{p_2}\c_\alpha: -(1+\alpha)x_1+\alpha x_3=0,\\
       T_{p_3}\c_\alpha: x_1+\alpha x_2=0,&&
       T_{p_4}\c_\alpha: -\alpha x_1-x_2+(1+\alpha)x_3=0.
       \end{array}
\end{displaymath}
       The condition~(II) implies that there are two smooth conics $\gamma_1:=\c_s, \gamma_2:=\c_t$
       for $s, t \not \in \{0, -1\}$, such that
       $p_0=T_{p_3}\c_s \cap T_{p_4}\c_s=T_{p_1}\c_t \cap T_{p_2} \c_t$.
       From the defining equations of the tangent lines, one sees that $t=-s$, $p_0=(t:1:1+t)$ and $t \not =0, \pm 1$. And $p_0$ clearly lies on $\l_{r_0s_0}: x_1+x_2 =x_3$.
\end{proof}
Fix $p_0=(t:1:1+t)$ with $t \not=0, \pm 1$. Then the defining equations of $\l_{p_0p_k}$ for $k=1, \ldots, 4$
are as follows
\begin{displaymath}
	\begin{array}{lll}
	\l_{p_0p_1}: -(1+t)x_2+x_3=0, && \l_{p_0p_2}: -(1+t)x_1+tx_3=0,\\
	\l_{p_0p_3}: x_1-tx_2=0,      && \l_{p_0p_4}: tx_1-x_2+(1-t)x_3=0.
	\end{array}
\end{displaymath}
\begin{proof}[Proof of \lemref{lem:gamma0}]
The existence of $\gamma_0$ can be shown by dimension counting.
It is clear that $\gamma_0$ is irreducible. Since the tangent line of $\gamma_0$ at $p_k$
is $\l_{p_0p_k}$,  the projection from $p_0$ gives a degree two map from $\gamma_0$ to  $\PP^1$ ramified at (at least) $4$ points $p_1, \ldots, p_4$. Since $p_a(\gamma_0)=1$,
this shows that $\gamma_0$ is a smooth elliptic curve.
\end{proof}
\begin{rem}
Actually, the curve $\gamma_0$  has the defining equation:
    \begin{align*}
    x_1^2[(1+t)x_2-x_3]+x_2^2[-(1+t)x_1+tx_3]+x_3^2(x_1-tx_2)=0.
   \end{align*}
\end{rem}

\begin{proof}[Proof of \propref{prop:p}]We have verified that $p_0, p_1, \ldots, p_4$ satisfy
conditions (I), (II), (II') in \lemref{lem:gamma0}.
So the sixth point $p$ satisfies condition (I)-(III) if and only if
$$p \not \in \l_{p_0p_1}\cup \ldots \cup l_{p_0p_4} \cup c_{p_0p_1p_2p_3p_4} \cup \gamma_0.$$
From now on, we assume $p$ satisfies conditions (I)-(III).

Consider the linear systems
$$\delta_1:=|\O_{\PP^2}(3)-p_0-p_1-p_2-p_3-p_4-p_2'-p_4'|,\ \delta_2 :=|\O_{\PP^2}(3)-p_0-p_1-p_2-p_3-p_4-p_1'-p_3'|.$$
Both have dimension $2$.
The point $p$ satisfies the condition (IV) if and only if
there exists cubics $\lambda_1 \in \delta_1$ and $\lambda_2 \in \delta_2$
such that both $\lambda_1$ and $\lambda_2$ have a singular point  at $p$.

Note that $|\delta_1|$ is generated by
  \begin{align*}
      \l_{p_0p_2}+\l_{p_1p_4}+\l_{p_3p_4}&: [-(1+t)x_1+tx_3](x_2-x_3)(x_1-x_2)=0,\\
      \l_{p_0p_4}+\l_{p_1p_2}+\l_{p_2p_3}&: [tx_1-x_2+(1-t)x_3]x_3x_1=0,\\
      \c_{p_0p_1p_2p_3p_4}+\l_{p_2p_4} &: [(-1+t^2)x_1x_2+x_1x_3-t^2x_2x_3](x_1-x_3)=0.
      \end{align*}
     Denote by $F(x_1, x_2, x_3)$ the Jacobian of three polynomials on the left of the above equations.
Similarly, $|\delta_2|$ is generated by
      \begin{align*}
      \l_{p_0p_1}+\l_{p_2p_3}+\l_{p_3p_4}&: [-(1+t)x_2+x_3]x_1(x_1-x_2)=0,\\
      \l_{p_0p_3}+\l_{p_1p_4}+\l_{p_1p_2}&: (x_1-tx_2)(x_2-x_3)x_3=0,\\
      \c_{p_0p_1p_2p_3p_4}+\l_{p_1p_3} &: [(-1+t^2)x_1x_2+x_1x_3-t^2x_2x_3]x_2=0.
      \end{align*}
Denote by $G(x_1, x_2, x_3)$ the Jacobian of three polynomials on the left of the above equations.
Then a direct calculation shows that
\begin{align*}
F(x_1,x_2,x_3)&=3[-(1+t)x_1+tx_3][tx_1-x_2+(1-t)x_3]Q(x_1,x_2,x_3),\\
G(x_1,x_2,x_3)&=3[-(1+t)x_2+x_3](x_1-tx_2)Q(x_1,x_2,x_3),
\end{align*}
where
$$Q(x_1,x_2,x_3)=[-(1+\alpha(t))x_1x_2+x_1x_3+\alpha(t)x_2x_3][-(1+\beta(t))x_1x_2+x_1x_3+\beta(t)x_2x_3].$$

It follows that there exist $\lambda_1 \in \delta_1$ and $\lambda_2 \in \delta_2$
such that both $\lambda_1$ and $\lambda_2$ have a singular point at $p$ if and only if
$Q(p)=0$, i.e. $p \in \c_{\alpha(t)}\cup \c_{\beta(t)}$.

Observe that $\c_{\alpha(t)} \cap \l_{p_kp_l}=\{p_k, p_l\}=\c_{\beta(t)}\cap \l_{p_kp_l}$ $(1 \le k \not=l \le 4)$, $c_{p_0p_1p_2p_3p_4}\cap c_{\alpha(t)}=c_{p_0p_1p_2p_3p_4}\cap c_{\beta(t)}=\{p_1, p_2, p_3, p_4\}$.
Therefore
the point $p$ verifies conditions (I)-(IV) if and only if
$p \in c_{\alpha(t)}\cup c_{\beta(t)}$ and $p \not\in \l_{p_0p_1}\cup \ldots \cup \l_{p_0p_4} \cup \gamma_0$.
\end{proof}

\section*{Acknowledgements}
 The first named author was supported by the National Natural Science Foundation of China (No.~11501019 and No.~11871084). The second named author was supported by the National Research Foundation of Korea(NRF) grant funded by the Korea government(MSIT) (No.~2010-0020413) and by Basic Science Research Program through the National Research Foundation of Korea(NRF) funded by the Ministry of Education (No. 2017R1D1A1B03028273).

\newpage

\noindent Yifan~Chen,\\
School of Mathematical Sciences, Beihang University,\\ Xueyuan Road No.~37, Beijing~100191, P. R. China\\
Email:~chenyifan1984@gmail.com\\\smallskip

\noindent YongJoo Shin,\\
Korea Institute for Advanced Study,\\
85 Hoegiro, Dongdaemun-gu, Seoul 02455, Republic of Korea\\
Email: haushin@kias.re.kr


\begin{thebibliography}{10}

\bibitem{RB85}
R.~Barlow.
\newblock \emph{Rational equivalence of zero cycles for some more surfaces with
  $p_g =0$}.
\newblock Invent. Math. \textbf{79}~(2) (1985), 303--308.

\bibitem{WBCPAV84}
W.~Barth, C.~Peters and A.~Van~de Ven.
\newblock \emph{Compact complex surfaces}, vol.~4 of \emph{Ergebnisse der
  Mathematik und ihrer Grenzgebiete (3) [Results in Mathematics and Related
  Areas (3)]}.
\newblock Springer-Verlag, Berlin (1984).

\bibitem{IB14}
I.~Bauer.
\newblock \emph{Bloch's conjecture for {I}noue surfaces with {$p_g=0$},
  {$K^2=7$}}.
\newblock Proc. Amer. Math. Soc. \textbf{142}~(10) (2014), 3335--3342.

\bibitem{IBFC12}
I.~Bauer and F.~Catanese.
\newblock \emph{Inoue type manifolds and {I}noue surfaces: a connected
  component of the moduli space of surfaces with {$K^2=7$}, {$p_g=0$}}.
\newblock In \emph{Geometry and arithmetic}, EMS Ser. Congr. Rep. Eur. Math.
  Soc., Z\"{u}rich (2012).
\newblock 23--56.

\bibitem{IBFCRP11}
I.~Bauer, F.~Catanese and R.~Pignatelli.
\newblock \emph{Surfaces of general type with geometric genus zero: a survey}.
\newblock In \emph{Complex and differential geometry}, vol.~8 of \emph{Springer
  Proc. Math.} Springer, Heidelberg (2011).
\newblock 1--48.

\bibitem{AB78}
A.~Beauville.
\newblock \emph{Surfaces alg\'{e}briques complexes}.
\newblock Soci\'{e}t\'{e} Math\'{e}matique de France, Paris (1978).
\newblock Avec une sommaire en anglais, Ast\'{e}risque, No. 54.

\bibitem{CCM07}
A.~Calabri, C.~Ciliberto and M.~Mendes~Lopes.
\newblock \emph{Numerical {G}odeaux surfaces with an involution}.
\newblock Trans. Amer. Math. Soc. \textbf{359}~(4) (2007), 1605--1632.

\bibitem{CMLP08}
A.~Calabri, M.~Mendes~Lopes and R.~Pardini.
\newblock \emph{Involutions on numerical {C}ampedelli surfaces}.
\newblock Tohoku Math. J. (2) \textbf{60}~(1) (2008), 1--22.

\bibitem{FC84}
F.~Catanese.
\newblock \emph{On the moduli spaces of surfaces of general type}.
\newblock J. Differential Geom. \textbf{19}~(2) (1984), 483--515.

\bibitem{FC99}
F.~Catanese.
\newblock \emph{Singular bidouble covers and the construction of interesting
  algebraic surfaces}.
\newblock In \emph{Algebraic geometry: {H}irzebruch 70 ({W}arsaw, 1998)}, vol.
  241 of \emph{Contemp. Math.} Amer. Math. Soc., Providence, RI (1999).
\newblock 97--120.

\bibitem{YC13}
Y.~Chen.
\newblock \emph{A new family of surfaces of general type with {$K^2=7$} and
  {$p_g=0$}}.
\newblock Math. Z. \textbf{275}~(3-4) (2013), 1275--1286.

\bibitem{YC15}
Y.~Chen.
\newblock \emph{Commuting involutions on surfaces of general type with
  {$p_g=0$} and {$K^2=7$}}.
\newblock Manuscripta Math. \textbf{147}~(3-4) (2015), 547--575.

\bibitem{DMP02}
I.~Dolgachev, M.~Mendes~Lopes and R.~Pardini.
\newblock \emph{Rational surfaces with many nodes}.
\newblock Compositio Math. \textbf{132}~(3) (2002), 349--363.

\bibitem{DG77}
D.~Gieseker.
\newblock \emph{Global moduli for surfaces of general type}.
\newblock Invent. Math. \textbf{43}~(3) (1977), 233--282.

\bibitem{IHMM79}
H.~Inose and M.~Mizukami.
\newblock \emph{Rational equivalence of $0$-cycles on some surfaces of general
  type with $p_g=0$}.
\newblock Math. Ann. \textbf{244}~(3) (1979), 205--217.

\bibitem{MI94}
M.~Inoue.
\newblock \emph{Some new surfaces of general type}.
\newblock Tokyo J. Math. \textbf{17}~(2) (1994), 295--319.

\bibitem{YLYS14}
Y.~Lee and Y.~Shin.
\newblock \emph{Involutions on a surface of general type with {$p_g=q=0$},
  {$K^2=7$}}.
\newblock Osaka J. Math. \textbf{51}~(1) (2014), 121--139.

\bibitem{MendesPardini01}
M.~Mendes~Lopes and R.~Pardini.
\newblock \emph{The bicanonical map of surfaces with {$p_g=0$} and {$K^2\geq
  7$}}.
\newblock Bull. London Math. Soc. \textbf{33}~(3) (2001), 265--274.

\bibitem{Pardini91}
R.~Pardini.
\newblock \emph{Abelian covers of algebraic varieties}.
\newblock J. Reine Angew. Math. \textbf{417} (1991), 191--213.

\bibitem{CR15}
C.~Rito.
\newblock \emph{Some bidouble planes with {$p_g=q=0$} and {$4\leq K^2\leq7$}}.
\newblock Internat. J. Math. \textbf{26}~(5) (2015), 1550035, 10.

\bibitem{CR18}
C.~Rito.
\newblock \emph{New surfaces with {$K^2=7$} and {$p_g=q\le2$}}.
\newblock Asian J. Math. \textbf{22}~(6) (2018), 1117--1126.

\end{thebibliography}
\end{document}